\documentclass[10pt]{amsart}
\usepackage{graphicx}
\usepackage{amsmath,amsthm}
\usepackage{amsfonts}
\usepackage{amscd}
\usepackage[all]{xy}
\usepackage{amssymb}
\usepackage[a4paper]{geometry}
\usepackage{subfigure}

\usepackage{tgpagella}

\newtheorem{theorem}{Theorem}[section]

\newtheorem{proposition}{Proposition}[section]
\newtheorem{lemma}{Lemma}[section]
\newtheorem{corollary}{Corollary}[section]
\newtheorem{definition}{Definition}[section]
\newtheorem{remark}{Remark}[section]

\newcommand{\norm}[1]{\left\Vert#1\right\Vert}

\newcommand{\Real}{\mathbb R}

\newcommand{\del}{\delta}

\newcommand{\oo}{\infty}
\newcommand{\esp}{\vspace{0.5cm}}

\newtheorem*{main}{Main Theorem}

\newcommand{\Fc}{\ensuremath{\mathcal{W}^c}}

\newcommand{\Fcs}{\ensuremath{\mathcal{W}^{cs}}}
\newcommand{\Fcu}{\ensuremath{\mathcal{W}^{cu}}}
\newcommand{\F}{\ensuremath{\mathcal{F}}}

\newcommand{\Ws}[1]{\ensuremath{W^s(#1)}}
\newcommand{\Wu}[1]{\ensuremath{W^u(#1)}}
\newcommand{\Wc}[1]{\ensuremath{W^c(#1)}}
\newcommand{\Wcs}[1]{\ensuremath{W^{cs}(#1)}}

\newcommand{\Wsloc}[1]{\ensuremath{W^s_{loc}(#1)}}
\newcommand{\Wuloc}[1]{\ensuremath{W^u_{loc}(#1)}}
\newcommand{\Wcloc}[1]{\ensuremath{W^c_{loc}(#1)}}
\newcommand{\Wcsloc}[1]{\ensuremath{W^{cs}_{loc}(#1)}}
\newcommand{\Wculoc}[1]{\ensuremath{W^{cu}_{loc}(#1)}}

\newcommand{\Wst}[1]{\ensuremath{\widetilde{\mathcal{W}}^s(#1)}}
\newcommand{\Wut}[1]{\ensuremath{\widetilde{\mathcal{W}}^u(#1)}}

\newcommand{\Wsl}[2][\epsilon]{\ensuremath{W^s(#2;#1)}}

\newcommand{\Wul}[2][\epsilon]{\ensuremath{W^u(#2;#1)}}

\newcommand{\Wcl}[2][\epsilon]{\ensuremath{W^c(#2;#1)}}

\begin{document}
\author{Pablo D. Carrasco}
\address{ICMC-USP, Avda Trabalhador S\~ao Carlense 400, S\~ao Carlos-SP, BR13566-590}
\email{pdcarrasco@gmail.com}
\thanks{The author is supported by FAPESP Project\# 2014/119262-0.}

\title{Symbolic Dynamics for Normally Hyperbolic Foliations}


\date{\today}

\begin{abstract}
We introduce some tools of symbolic dynamics to study the hyperbolic directions of partially hyperbolic diffeomorphisms, emulating the well known methods available for uniformly hyperbolic systems.
\end{abstract}

\maketitle

\section{Introduction and Main Results.}

It is somewhat surprising that large classes of smooth systems can be accurately described by a symbolic model. In these cases, the extra tools of symbolic dynamics provide a concrete framework to study the more complicated differentiable system, allowing instead to work with concrete zero-dimensional models.  A prototypical example here are the uniformly hyperbolic diffeomorphisms, where the existence of Markov Partitions \cite{SinaiMarkov},\cite{BowenMarkov} permit to reduce many aspects of their theory to properties of subshifts of finite type. The celebrated Sinai-Ruelle-Bowen theory for  Axiom A attractors, one of the fundamental pieces in smooth ergodic theory, originated in this reduction. We refer the reader to \cite{EquSta} for an introduction to the techniques and background material.

A similar structure is valid for Anosov flows \cite{RatnerMarkov} or more generally, hyperbolic flows \cite{SymbHyp}. In this case R. Bowen showed that if $\phi_t:M\rightarrow M$ is a flow on compact manifold and $\Omega$ is a non-trivial basic set, then $\phi_t:\Omega\rightarrow \Omega$ is finitely covered by the suspension of a subshift of finite type under a Lipschitz roof function. This means the following: in the above situation there exist a subshift of finite type $\Sigma_A$ with shift map $\sigma$, a Lipschitz function $\rho:\Sigma_A\rightarrow \Real_{>0}$ such that for the space
$$
S=\{(\underline{a},t)\in \Sigma_A\times \Real: t\in[0,\rho(\underline{a})]\}/(\sigma(\underline{a}),0)\sim (\underline{a},\rho(\underline{a}))\ 
$$
one can find a  continuous function $h:S\rightarrow \Omega$ satisfying 
\begin{enumerate}
	\item $h$ is onto.
	\item There exists $K>0$ such that for every $x\in \Omega$, $\# h^{-1}(x)\leq K$.
	\item $h$ is one to one on a $\mathcal{G}_{\delta}$ set $X\subset \Sigma_A$, and moreover $X$ is of full probability.
	\item $h\circ s_t=\phi_t\circ h$ for every $t$, where $s_t:S\rightarrow S$ is the natural horizontal flow\ ($s_t[(\underline{a},t')]=[(\underline{a},t+t')]$).
\end{enumerate}

As in the case of hyperbolic diffeomorphisms the above construction permits to establish very subtle properties of hyperbolic flows, for example an extension of the Sinai-Ruelle-Bowen theory \cite{BowenRuelle}, the Central Limit Theorem  for geodesic flows on manifolds of negative sectional curvature \cite{RatnerCLT}, of the decay of matrix coefficients for geodesic flows on negatively curved surfaces \cite{Dolgodecay}, just to cite a few.

It is important to notice that Bowen's construction establishes the symbolic model for the flow, and not for a fixed map in the flow (say, the time one-map). The time-t maps of the flow are not hyperbolic, and a simple entropy argument shows that one cannot hope for a symbolic representation of them as in the case of hyperbolic diffeormorphisms. 

The purpose of this paper is to present a symbolic model for a generalization of hyperbolic systems, the so called partially hyperbolic diffeomorphisms. In those we allow, besides the hyperbolic directions, a third center direction where the dynamics is only known to be dominated by the hyperbolic ones. Time $t$-maps of hyperbolic flows are partially hyperbolic, provided that $t$ is not zero, but there are other totally different systems which are also partially hyperbolic, as the ergodic linear automorphisms of the Torus, or the so called skew-products. We provide the precise definition in the next section and refer the reader to \cite{PesinLect} for an introduction to Partially Hyperbolic Dynamics.

We will assume that, as is the case for the majority of the known examples, the center direction integrates to a foliation with $f$-invariant leaves. Although our main interest is the study of the hyperbolic directions, this transverse dynamics is in principle not well defined and we need to make use of the holonomy pseudo-group of the center foliation to gain some control on its structure. The presence of this pseudo-group, on the other hand, is one of the main difficulties for the construction. Another problem appears due to in general, the partially hyperbolic diffeomorphism does not fixes the center foliation. In any case, we are able to stablish a Markov structure for the transverse hyperbolic directions (see next section for the definition of $\Fcs_{loc},\Fcu_{loc}$).

\begin{main}
Let $f:M\rightarrow M$ be a partially hyperbolic diffeomorphism with (plaque expansive) center foliation $\Fc$. Then given $\delta>0,\rho>0$ there exist a family $\mathcal{D}=\{D_1,\ldots, D_N\}$ of $s+u$ dimensional (embedded) discs transverse to \Fc, a family $\mathcal{R}=\{R_1,\ldots,R_s\}$\ and a function $af:\mathcal{R}\rightarrow 2^{\mathcal{D}}$\ satisfying the following.

\begin{enumerate}
\item The family $\mathcal{D}$ is pairwise disjoint.
\item Each $R_{\mu}$ is contained in some disc $D_i$ away from the boundary, and with respect to the relative topology of $D_i$ it holds $R_{\mu}=cl\, int R_{\mu}$. Moreover $diam R_{\mu}<\rho$.
\item If $R_{\mu},R_{\nu}$ are contained in the same disc $D_i$ and their relative interiors share a common point, then $R_{\mu}= R_{\nu}$.
\item $M=\cup_{\mu=1}^s sat^c(R_{\mu},\delta)$, where $sat^c(R_{\mu},\delta)$ consists of all the center plaques of size $2\delta$ centered at points of $R_{\mu}$.
\item Suppose that $R_{\mu}\subset D_i,R_{\nu}\in D_{i'}$ satisfy  $D_{i'}\in af(R_{\mu})$. If  $x\in int\, R_{\mu}, \phi_{\mu,i'}^+(x)\in int\, R_{\nu}$, then it holds
\begin{enumerate}
\item $\phi_{\mu,i'}^+(\Wst{x,R_{\mu}})\subset \Wst{\phi_{\mu,i'}^+(x),R_{\nu}}$.
\item $\phi_{\mu,i'}^{-}(\Wut{\phi_{\mu,i'}^+(x),R_{\nu}})\subset \Wut{x,R_{\mu}}$
\end{enumerate}
where $\phi_{\mu,i'}^+=proj_{D_{i'}}\circ f,\phi_{\mu,i'}^-=proj_{D_i}\circ f^{-1}, proj_{D_i}:sat^c(D_i,\delta)\rightarrow D_i$ is the map that collapses center plaques and
$$
\Wst{x,R_{\mu}}=\Wcsloc{x}\cap R_{\mu}, \quad \Wut{x,R_{\mu}}=\Wculoc{x}\cap R_{\mu}.
$$
\end{enumerate}
\end{main}

Fix families $\mathcal{D}=\{D_i\}_{i=1}^N,\mathcal{R}=\{R_{\mu}\}_{\mu=1}^s$ as the ones given in the previous theorem. Denote $\Gamma(\mathcal{R})=\cup_{\mu}R_{\mu}$ and consider the pseudo-group $\mathcal{G}(\mathcal{R})$ of homeomorphisms of $\Gamma(\mathcal{R})$ generated by $\{\phi_{\mu,i'},\phi_{\mu,i'}^-:i'\in af(R_{\mu})\}$. Consider the subshift of finite type $\Sigma_S$ determined by the matrix $S$ where
$$
S_{\mu,\nu}=1\Leftrightarrow \nu\subset D_{i'},i'\in af(R_{\mu})\text{ and } \phi_{\mu,i'}^+(int R_{\mu})\cap int R_{\nu}\neq\emptyset.
$$

\begin{corollary}\label{coroA}
There exists a continuous surjective map $h:\Sigma_S\rightarrow \Gamma(\mathcal{R})$ that semiconjugates the action of the shift $\sigma:\Sigma_S\rightarrow\Sigma_S$ with the action of the pseudo-group $\mathcal{G}(\mathcal{R})\curvearrowright\Gamma(\mathcal{R})$.
\end{corollary}
 
The previous Corollary can be seen as an analog of Ratner-Bowen construction for hyperbolic flows, although the map $h$ will not be bounded-to-one due to the fact that we have to take into account the center holonomy, and this is much worse behaved in general than in the case of $\mathbb{R}$-actions.

\esp

The remainder of the article is divided into three sections. In the first we provide some background material and notation that will be used in the sequel. The third Section is devoted to the proof of the main theorem, and in the final one we discuss the coding associated to this transverse model.

\section{Preliminaries}

In this section we collect some definitions and establish some notation used in the sequel. Throughout the paper, $M$ will denote a closed (compact, without boundary) differentiable manifold. We are interested in the following generalization of Anosov systems.

\begin{definition}
A $\mathcal{C}^1$\ diffeomorphism  $f: M \rightarrow  M$\ is \emph{partially hyperbolic} if there exists a continuous splitting of the tangent bundle of the form  $$TM=E^u\oplus E^c\oplus E^s$$\ where both bundles $E^s,E^u$\ have positive dimension, and a continuous Riemannian metric $\norm{\cdot}$ on $M$ such that
\begin{enumerate}
  \item all bundles $E^u,E^s,E^c$\ are $df$-invariant.
  \item For every $x\in M$, for every unitary vector $v^{\sigma}\in E^{\sigma}_x, \sigma=s,c,u$,
  \begin{gather*}
  \norm{d_xf(v^{s})}<1<\norm{d_xf(v^{u})}\\
  \norm{d_xf(v^{s})}< \norm{d_xf(v^{c})}<\norm{d_xf(v^{u})}
  \end{gather*}
\end{enumerate}
\end{definition}

The bundles $E^s,E^u,E^c$\ are called the \emph{stable, unstable} and \emph{center} bundle respectively. The well known \emph{Stable Manifold Theorem} implies that the continuous bundles $E^s,E^u$ are integrable to $f$-invariant continuous foliations $W^s,W^u$, whose leaves are of class $C^1$. These are called the stable and unstable foliations, respectively. See chapter $4$ of \cite{HPS} for the proof.

On the other hand, the center bundle is not always integrable (Cf. \cite{SmaleBull},\cite{Nodyncoh}), although it is almost always satisfied in the examples. As one of our main objectives is providing a model for the foliation tangent to $E^c$, the center foliation, we will assume that this bundle is integrable and in fact, to avoid some technical subtleties, we will assume something (seemingly) stronger.

\begin{definition}
A partially hyperbolic map $f:M\rightarrow M$ is \emph{dynamically coherent} if the bundles $E^c$,$E^{cs}=E^c\oplus E^s$,$E^{cu}=E^c\oplus E^u$ are integrable to continuous $f$-invariant foliations $\Fc, \Fcs, \Fcu$
\end{definition}

\textbf{Convention:} From now on, we will omit dynamically coherent when referring to partially hyperbolic maps.

\esp

For a point $x\in M$, a positive number $\gamma>0$\ and $\sigma\in \{s,u,c,cs,cu\}$ we will denote by $W^{\sigma}(x;\gamma)$\ the open disc of size $\gamma$\ inside the leaf $W^{\sigma}(x)$, where distances are measured with the corresponding induced metric. Moreover, if $A\subset M$ we denote
\begin{align}
W^{\sigma}(A;\gamma)&:=\cup_{x\in A}W^{\sigma}(x;\gamma),\quad \sigma\neq c\\
sat^c(A)&:=\cup_{x\in A}W^{c}(x;\gamma).
\end{align}

In the definition or partial hyperbolicity the metric used can be assumed such that the bundles $E^c,E^s$\ and $E^u$\ are mutually orthogonal \cite{AdaptedMet}.  It follows that we have \emph{local product structure}: there exists some $r_0>0$\ such that for $0<r\leq r_0$,
$$d(x,y)<r\Rightarrow \exists\ z\text{ s.t. } H_x^r\cap V_y^r\supset \Wcl[r]{z},$$ 
where $H_x^r=\Wul[2r]{\Wcl[2r]{x}},V_x^r=\Wsl[2r]{\Wcl[2r]{x}}$. See Section 7 of \cite{HPS}.

\begin{definition}
The local plaque of $\F^{\sigma}$ centered at $x$ is defined as
$$
W^{\sigma}_{loc}(x):=\begin{cases}
W^{\sigma}(x,2r_0)&\quad \sigma\in\{cs,cu\}\\
W^{\sigma}(x,r_0)&\quad \sigma\in\{c,s,u\}.
\end{cases}
$$
\end{definition}

One of the main consequences of dynamical coherence if the property of shadowing.

\begin{definition} Let $f: M\rightarrow M$\ be partially hyperbolic.
\begin{enumerate}
\item A sequence $\underline{x}=\left\{ x_n \right\}_{-N}^N $ where $N \in \mathbb{N}\cup \{\infty\}$\ is called a $\delta$-pseudo orbit for $f$ if for every $n=-N,\ldots , N-1$, $d(fx_n,x_{n+1})\leq \delta$. If furthermore $f(x_n)\in \Wcl[\delta]{x}$ for every $n=-N,\ldots,N-1$ we say that $\underline{x}$ is a central $\delta$-pseudo orbit. Center stable and center unstable pseudo orbits are defined similarly using the foliations \Fcs,\Fcu.

\item For $\epsilon>0$, we say that the pseudo orbit $\underline{y}=\left\{ y_n \right\}_{-N}^N$\ $\epsilon$-shadows the pseudo orbit $\underline{x}=\left\{ x_n \right\}_{-N}^N$\ if $d(x_n,y_n)<\epsilon$ for every $n=-N,\ldots , N-1$.
\end{enumerate}
\end{definition}

\begin{theorem}[Shadowing]\label{shadowing}
Let $f: M\rightarrow M$\ be partially hyperbolic. Then there exist constants $C>1,\delta_0>0$ such that
for $0<\delta\leq \delta_0$ any $\delta$ pseudo-orbit\ can be $C\delta$-shadowed by a $C\delta$-central pseudo-orbit. 
\end{theorem}

This theorem is a generalization of the classical shadowing theorem for Anosov maps due to Hirsch-Pugh-Shub, and appears as Theorem 7A-2 in \cite{HPS} in a slightly different formulation, although the version above follows directly from their arguments.

In principle, there could be more than one bi-infinite central pseudo-orbits shadowing a given one. To remedy this we will work with plaque expansive systems.

\begin{definition}
Let $f:M\rightarrow M$ be a partially hyperbolic set. We say that:
\begin{enumerate}
\item The foliation \Fc is plaque expansive if there exists $e>0$ such that if
$\underline{x}=\left\{ x_n \right\}_{-\oo}^{\oo},\underline{y}=\left\{ y_n \right\}_{-\oo}^{\oo}$ are central $e$-pseudo orbits satisfying $d(x_n,y_n)$ for every $n$, then $x_0\in \Wc{y_0;e}$. In this case we usually say that the map $f$ is plaque expansive
\item The foliation \Fcs\ is plaque expansive if there exists $e>0$ such that if
$\underline{x}=\left\{ x_n \right\}_{0}^{\oo},\underline{y}=\left\{ y_n \right\}_{0}^{\oo}$ are center stable $e$-pseudo orbits satisfying $d(x_n,y_n)$ for every $n$, then $x_0\in \Wcs{y_0;e}$. Similarly for \Fcu.
\end{enumerate}
\end{definition} 

Plaque expansiveness holds in every known example, and is guaranteed if $\Fc$ is $C^1$, or if $f$ is a center isometry, i.e. $\norm{df|E^c}=1=m(df|E^c)$. It is also $C^1$-stable: if $f$ is plaque expansive then there exists a $C^1$ neighbourhood $U$ of $f$ such that every $g\in U$ is plaque expansive. See Chapter 6 in \cite{HPS}.

\begin{lemma}\label{plaqueexpother}
If $f$ is plaque expansive, the foliations $\Fcs,\Fcu$ are also plaque expansive.
\end{lemma}

The proof is not hard.

\subsection{Adapted families of discs.}

We will make extensive use of transverse discs to the center foliation. 

\begin{definition}
For $0<\epsilon<r_0$ we say that $D\subset M$ is an adapted disc of size $\epsilon$ if there exists an embedding 
$h_D:D^{s+u}(\epsilon):=\{p\in R^{s+u}:\|p\|< \epsilon\}\rightarrow M$ satisfying
\begin{enumerate}
\item $im h=D$ and $D$ is transverse to \Fc.
\item For all $p, \frac{1}{K_1}\leq m(dh_p), \|dh_p\| \leq K_1$
where $K_1\approx 1$ is a constant. 
\item If $x,y\in D$, $\Wcloc{x}\cap\Wcloc{y}\neq \emptyset$ implies $x=y$.
\end{enumerate}
\end{definition}

\begin{remark} The constant $K_1$ is chosen to guarantee that $h$ is almost an isometry, and only depends on $M$ and $f$.  The center of $D$ is $h(0)$.
\end{remark}

For $D$ as above we have a projection $proj_D:sat^{c}(D,r_0)\rightarrow D$
given by 
\begin{equation}
proj_D(y)=\Wcloc{y}\cap D,
\end{equation}
and if $z\in D$ we define the sets
\begin{align}
\Wut{z,D}:=H_z\cap D\\
\Wst{z,D}:=V_z\cap D.
\end{align}

Observe that if $D$ is an adapted disc of size $\epsilon\geq 3r$ centred at $x$ and $y\in D, d(x,y)<2r$,
there exist uniquely defined points $y^u_D\in \Wut{x,D},y^s_D\in\Wst{x,D}$\ such that $y=D\cap V_{y^u_D}^r\cap H_{y^s_D}^r$.
The map $\Psi_x^D$ given by
$$
\Psi_x^D(y)=(y^u_D,y^s_D)
$$
is an open embedding, and thus define a continuous system of coordinates on $D$. We write  $<\cdot,\cdot>_D:\Wut{x,D}\times\Wst{x,D}\rightarrow D$\ for the inverse of $\Psi^D$.

\esp

\begin{definition} Let $0<\delta<\epsilon<r_0$. We say that $\{D_i\}_{i=1}^N$ is an $(\epsilon,\delta)$ adapted family if the following conditions hold
\begin{enumerate}
\item Each $D_i$ is an $\epsilon$-adapted disc.
\item If $i\neq j$ then $D_i\cap D_j=\emptyset$.
\item For each $i$ there exist open sub-discs $B_i\subset E_i \subset D_i$ of the form
\begin{itemize}
\item $B_i=h_{D_i}(\delta)$.
\item $E_i=h_{D_i}(\frac{\epsilon}{2})$
\end{itemize}
\item $M=\cup_{i=1}^N sat^c(B_i,\delta)$.
\end{enumerate}
\end{definition}

\textbf{Convention:} Recall Theorem \ref{shadowing}. For the previous definition we will assume that $\delta$ is chosen so that
\begin{enumerate}
\item $8C\delta<<\epsilon$.

\item Every $(3+Lip(f|\Fc))\delta$-pseudo-orbit can be $C\delta$-shadowed by a pseudo-orbit subordinated to \Fc.
\end{enumerate}

\begin{lemma} For every $\epsilon>0$ sufficiently small and $0<\delta<\frac{\epsilon}{8C}$ there exists an $(\epsilon,\delta)$ adapted family.
\end{lemma} 

\begin{proof}
This is easy consequence of the existence of nice (finite) atlases of arbitrarily small diameter for the center foliation\footnote{A foliated atlas $\{U\}_k$ is nice if whenever $U_k\cap U_{k'}\neq \emptyset$ there exists a distinguished open set containing $cl(U_k\cup U_{k'})$.}. See Lemma 1.2.17 in \cite{FoliationsI}.
\end{proof}

For later use, we also record the following simple remark.

\begin{lemma}\label{projectain}
There exists $\epsilon_0>0$ such that every $(\epsilon,\delta)$ adapted family $\{D_i\}_{i=1}^N$ with $0<\epsilon\leq \epsilon_0$ satisfies the following:
$$
x\in M, d(x,B_i)<C\delta\Rightarrow \#E_i\cap \Wcl[C\delta]{x}=1.
$$
\end{lemma}

\esp

\subsection{Shift Spaces}

Let $k$ be a natural number. If $A$ is a $0-1$ $k\times k$ square matrix we denote by $\Sigma_A,\Sigma_A^+$ the two-sided and one-sided subshifts of finite type determined by $A$, namely
\begin{align}
\Sigma_A&:=\{\underline{x}\in \{1,\ldots,k\}^{\mathbb{Z}}:A_{x_n,x_{n+1}}\forall n\in\mathbb{Z} \}\\
\Sigma_A^+&:=\{\underline{x}\in \{1,\ldots,k\}^{\mathbb{N}}:A_{x_n,x_{n+1}}\forall n\in\mathbb{N} \}.
\end{align} 

Equipped with the product topology, these are compact metrizable spaces where compatible metrics are given by $d(\underline{x},\underline{y})=\frac{1}{2^{N(\underline{x},\underline{y})}}$ with
\begin{equation}
N(\underline{x},\underline{y})=\begin{cases}sup\{n:x_i=y_i\ \forall\ |i|<n\}&\text{ for }\underline{x},\underline{y}\in \Sigma_A\\
sup\{n:x_i=y_i\ \forall\ 0\leq i<n\}&\text{ for }\underline{x},\underline{y}\in \Sigma_A^+.
\end{cases}
\end{equation}

It is well known that both $\Sigma_A,\Sigma_A^+$ are zero-dimensional perfect compact sets, and thus homeomorphic to Cantor sets. The corresponding shift map in each of them will be denoted simply by $\sigma$.

The two-sided subshift has a natural bracket structure: for sequences $\underline{a},\underline{b}$\ with $a_0=b_0$\ their bracket is the sequence $\underline{c}$\ given by
$$c_{n}=\begin{cases}
a_n&\text{if }n\geq 0 \\
b_{n}&\text{if }n<0
\end{cases}.$$
To understand this structure we introduce the local stable set and the local unstable set of a point $\underline{a}\in \Sigma_{A}$ as

\begin{align}
W^s_{loc}(\underline{a})&=\{\underline{b}\in\Sigma_{A^1}:d(\sigma^n\underline{a},\sigma^n\underline{b})\leq 1/3\ \forall\ n\geq 0\}\\
W^u_{loc}(\underline{a})&=\{\underline{b}\in\Sigma_{A^1}:d(\sigma^n\underline{a},\sigma^n\underline{b})\leq 1/3\ \forall\ n\leq 0\}.
\end{align}

Observe that 
$$
W^s_{loc}(\underline{a})=\{\underline{b}\in\Sigma_{A^1}:a_n=b_n\ \forall n\geq 0\}
$$
$$
W^u_{loc}(\underline{a})=\{\underline{b}\in\Sigma_{A^1}:a_n=b_n\ \forall n\leq 0\},
$$
and in particular $\langle\underline{a},\underline{b}\rangle=W^s_{loc}(\underline{a})\cap W^u_{loc}(\underline{b})$.

\esp

\section{Complete Markov Transversal}

Throughout this section we will work with a fixed plaque expansive partially hyperbolic diffeormorphism $f: M\rightarrow M$ and establish the existence of a complete transversal to its center foliation $\Fc_f$ consisting of closed sets with an additional Markov property (Cf. Theorem A).  Precision in this stage seems crucial, so for convenience of the reader we have divided the construction in several parts.

\subsection{Symbols and Transition Matrices.}

We fix an $(\epsilon,\delta)$-adapted family $\{D_i\}_{i=1}^N$ satisfying the conclusion of Lemma \ref{projectain}. Subdivide each $B_i$ into finitely many (relatively) open sets $\Delta^i_1,\ldots \Delta^i_{r_i}$ with the following property:
given $i\in\{1,\ldots N\},\alpha\in\{1,\ldots r_i\}$ there exist $j(i,\alpha)$, $k(i,\alpha) \in  \{1,\ldots N\}$ so that
$$
f(\Delta^i_{\alpha})\subset sat^c(B_j;\delta)
$$
$$
f^{-1}(\Delta^i_{\alpha})\subset sat^c(B_k;\delta).
$$
Let $\nu=\max_{i,\alpha}\{diam(\Delta^i_{\alpha}),diam(f(\Delta^i_{\alpha})),diam(f^{-1}(\Delta^i_{\alpha}))\}$; note that $\nu$ can be taken arbitrarily small, and we assume that $3Lip(f)\nu<\delta$. The previous consideration allows us to define a map $\phi: \sqcup_{i,\alpha} \Delta^i_{\alpha}\rightarrow \cup_{i} B_i$ simply by
$$
\phi(x)=\Wcl[\delta]{fx}\cap B_j \text{ if }x\in \Delta^i_{\alpha}.
$$

Nonetheless, the presence of the holonomy pseudo-group of $\Fc$ forces a complicated combinatorics which make $\phi$ very ill behaved, even in simple examples. Instead, we will use symbolic dynamics to control the center pseudo-group. We proceed as follows. On each $\Delta^i_{\alpha}$ choose a point $x^i_{\alpha}$, and consider
$I=\{x^i_{\alpha}:i,\alpha\}$. Define the 0--1 matrix $A$ with indices on $I$ by the rule that $A_{x^i_{\alpha},x^{i'}_{\alpha'}}=1$ iff either

\begin{enumerate}
\item $f(\Delta^i_{\alpha})\subset sat^c_{\delta}(B_{i'})$, or
\item exists $\Delta^{i'}_{\beta}\text{ s.t. }f^{-1}(\Delta^{i'}_{\beta})\subset sat^c_{\delta}(B_{i})\text{ and } \ proj_{D_i}(f^{-1}(\Delta^{i'}_{\beta}))\cap  \Delta^i_{\alpha}\neq\emptyset$
\end{enumerate}


We remark the following.
\begin{lemma}\label{combinatoria} For every $i,\alpha,\alpha'$ we have
$$
A_{x^{i}_{\alpha},x^{j(i,\alpha)}_{\alpha'}}=1=A_{x^{k(i,\alpha)}_{\alpha'},x^i_{\alpha}}.
$$
\end{lemma}

\esp

We consider the subshift of finite type $\Sigma_{A}$ and observe that every $\underline{a}\in \Sigma_{A}$ defines a $(3+Lip(f|\Fc))\delta$-pseudo orbit, hence it is $C\delta$-shadowed by a $C\delta$-pseudo orbit $\underline{a}'$ that preserves \Fc. By  Lemma \ref{projectain}, the point
\begin{equation}\label{deftheta}
\theta(\underline{a}):=\Wcl[C\delta]{a'_0}\cap E_i,
\end{equation}
where $a_0\in D_i$,  is well defined and thus we have a map $\theta=\theta_{A}:\Sigma_{A}\rightarrow \cup_i E_i$. 

\esp

\begin{proposition}
The map  $\theta :\Sigma_{A}\rightarrow\cup_1^n E_i$\ is continuous.
\end{proposition}

\begin{proof}
We need the following lemma.

\begin{lemma}
Assume that $(\underline{a}^k)_{k\in \mathbb{N}}$\ is a sequence in $\Sigma_{A}$ that converges  to $\underline{a}$. Then there exists a subsequence $(\underline{a}^k)_{k\in S}$\ such that for every $n$,
$$ (a^k)'_n \xrightarrow[k\in S]{} a_n'.$$
\end{lemma}

\begin{proof}
We recall that the construction of $\underline{a}'$\ goes by finding for each $N>0$\ a finite pseudo-orbit $\{z_{-N}^N(\underline{a}),\ldots$ $z_0^N(\underline{a}),\ldots z_N^N(\underline{a})\}$\ which shadows the finite segment $\{a_{-N},\ldots,a_0,\ldots , a_{N}\}$\ and then passing to a converging subsequence $z_n^{N_j}(\underline{a})\xrightarrow[j\mapsto\infty]{} a_n'$ (Cf. Theorem 7A-2 in \cite{HPS}).

\esp

Now, since $\underline{a}^k\xrightarrow[k\mapsto\infty]{} \underline{a}$, there exists $k_1$\ such that $\forall k\geq k_1\  a_n^k=a_n$\ if $|n|\leq N_1$ . Thus $\forall k\geq k_1$\ we can take $$z_n^{N_1}(a^k)=z_n^{N_1}(a) \ \text{if } |n|\leq N_1 $$
By induction, we can find $k_j>k_{j-1}$\
    \begin{equation*}
    \forall k\geq k_j , \ z_n^{N_j}(a^n)=z_n^{N_j}(a) \ \text{if } |n|\leq N_j
    \end{equation*}
and from this follows $$ z_n^{N_j}(a^{k_j})\xrightarrow[j\mapsto\infty]{} a_n' $$
\end{proof}

Back to the proof of the proposition, suppose that there exist two sequences $(\underline{a}^k)_k , (\underline{b}^k)_k$\ in $\Sigma_{A}$ converging to the same limit $\underline{a}$, but with $$x=\lim_k \theta(\underline{a}^k)\neq \lim_k \theta(\underline{b}^k)=y.$$

By the previous lemma we can assume with no loss of generality that for every $n$,
\begin{equation}\label{conv1}
(a^k)'_n\xrightarrow[k\mapsto\infty]{}a'_n
\end{equation}
\begin{equation}\label{conv2}
(b^k)'_n\xrightarrow[k\mapsto\infty]{}a'_n.
\end{equation}
Consider the adapted disc $D_i$ that contains $x,y$ and denote by $proj: sat^c(D_{i})\rightarrow D_i$ the projection map: $proj$ is continuous. Thus by \eqref{conv1},\eqref{conv2}

$$x=\lim_{k\rightarrow \infty}proj((a^k)'_0)=proj(a'_0)=\lim_{k\rightarrow \infty}proj((b^k)'_0)=y,$$
contrary to what we assumed. Since $\cup_{i=1}^{N}E_i$ is pre-compact we conclude that $\theta$\ is continuous.
\end{proof}

We now study the image of the map $\theta$. First of all we have the following.

\begin{proposition}\label{sobre}
The image of $\theta$\ is compact in $\cup_1^N E_i$\ and covers $\cup_1^N B_i$.
\end{proposition}

\begin{proof}
The image is $\theta$\ is clearly compact due to the previous Proposition. To prove the second part we fix $x\in B_i$\ and define
an element $\underline{a}\in\Sigma_{A}$\ by the following procedure. Let $i_0,\alpha_0$ such that $x\in \Delta^{i_0}_{\alpha_0}$,
and set $a_0=x^{i_0}_{\alpha_0},x_0=x$. Now we inductively define $a_n,x_n$\ as follows.

\begin{itemize}
\item Assume first that $n>0$, and that we had determined $x_k,a_k$\ for $k=0,\ldots n$, with 
$a_n=x^{i_n}_{\alpha_n},x_n\in \Delta^{i_n}_{\alpha_n}$. Define $x_{n+1}=proj_{D_{j(i_n,\alpha_n)}}(fx_n)$ and choose $\Delta^{j(i_n,\alpha_n)}_{\alpha_{n+1}}$
containing $x_{n+1}$. Finally define $a_{n+1}=x^{j(\alpha_n)}_{\alpha_{n+1}}$. Note that $A_{a_n,a_{n+1}}=1$. 

\item For $-n<0$\ and assuming we have $a_{-k},x_{-k}$\ for $k=0,\ldots,n$\ note that $f^{-1}\Delta^{i_{-n}}_{\alpha_{-n}}\subset sat^c(B_{k(i_{-n},\alpha_{-n})},\delta)$, hence denoting $i_{-n-1}=k(i_{-n},\alpha_{-n})$ there exist $\alpha_{-n-1}$ and $x_{-n-1}\in\Delta^{i_{-n-1}}_{\alpha_{-n-1}}$ satisfying $proj_{D_{i_{-n}}}(fx_{-n-1})=x_{-n}$. Labelling $a_{-n-1}=x^{i_{-n-1}}_{\alpha_{-n-1}}$ we get by Lemma \ref{combinatoria} that $A_{a_{-n-1},a_{-n}}=1$.
\end{itemize}

In the end we have constructed two bi-infinite sequences $\underline{x},\underline{a}$ satisfying

\begin{enumerate}
\item the sequence $\underline{a}\in\Sigma_{A}$ \
\item The sequence $\underline{x}$\ is a central $(2\nu+Lip(f|\Fc))\delta$-pseudo orbit.\
\item For every $n,\  d{(a_n,x_n)}\leq \nu$.
\end{enumerate}

By plaque expansiveness we conclude $x=x_0=\theta(\underline{a})$.
\end{proof}

\esp

\begin{remark}\label{sobresymbol}
For later use we record that in the previous Proposition the sequence $\underline{a}=(x^{i_n}_{\alpha_n})_n$ constructed is such that for every $n$,
$proj(f\Delta_{x^{i_n}_{\alpha_n}})\cap \Delta_{x^{i_{n+1}}_{\alpha_{n+1}}}\neq \emptyset$. 
\end{remark}

\subsection{Rectangles and Markov Structure}

As explained in the first section, the local product structure allows us to define a bracket for points in $\cup_{i=1}^N E_i$, where for $x,y\in E_i$\ we have $\langle x,y\rangle_D=proj_{D_i}(\Wsloc{x}\cap\Wculoc{y})$. We now relate this bracket structure with the one of $\Sigma_{A}$

\begin{proposition}\label{brapre}
Let $\underline{a}\in \Sigma_{A}$ such that $a_o\in D_i$. Then

\begin{enumerate}
\item $\theta(\Wsloc{\underline{a}})\subset proj_{D_{i}}(\Wcsloc{\theta(\underline{a})}$ \
\item $\theta(\Wuloc{\underline{a}})\subset proj_{D_{i}}(\Wculoc{\theta(\underline{a})}$
\end{enumerate}
\end{proposition}

\begin{proof}
Take $\underline{b}\in W^s_{loc}(\underline{a})$.  Since $a_n=b_n\ \forall n\geq 0$\ we have $$d(a_n',b_n')\leq 2C\delta\quad \forall n\geq 0.$$ 
 
Theorem 6.1 of \cite{HPS} implies then that $b_0'\in\Wcsloc{a_0'}$, hence the result. The second part is analogous.
\end{proof}

\esp

\begin{corollary}\label{conmuta}
The map $\theta$\ is bracket preserving.
\end{corollary}

\begin{proof}
Let $\underline{a},\underline{b}\in\Sigma_{A}$\ with $a_0=b_0\in D_i$. Then $\langle\theta\underline{a},\theta \underline{b}\rangle)_{D_i}$\ is well defined and
\begin{equation*}
\begin{split}
\theta(\langle \underline{a},\underline{b}\rangle) &= \theta(W^s_{loc}(\underline{a})\cap W^u_{loc}(\underline{b}))\subset proj_{D_i}(\Wcsloc{\theta \underline{a}})\cap proj_{D_i}(\Wculoc{\theta \underline{b}})=\{\langle\theta \underline{a},\theta \underline{b}\rangle_{D_i}\}.
\end{split}
\end{equation*}
\end{proof}

\begin{definition}
For $i=1,\ldots N$, a subset $R\subset D_i$ will be called a rectangle if it is invariant under the bracket $\langle \cdot, \cdot \rangle_{D_i}$. Similarly, a subset $R'\subset \Sigma_{A}$ will be called a rectangle if it is invariant under the bracket.
\end{definition}

One verifies easily for arbitrary $x_{\alpha_0}^{i_0},\ldots x_{\alpha_r}^{i_r}\in I$ the set
\begin{equation}
[x_{\alpha_0}^{i_0},\ldots x_{\alpha_r}^{i_r}]:=\{\underline{a}\in\Sigma_{A} : a_0=x_{\alpha_0}^{i_0},\ldots a_r=x_{\alpha_r}^{i_r}\}
\end{equation}
is a compact rectangle, hence by Proposition \ref{brapre} the set 
\begin{equation}
C_{x_{\alpha_0}^{i_0},\ldots x_{\alpha_r}^{i_r}}=\theta[x_{\alpha_0}^{i_0},\ldots x_{\alpha_r}^{i_r}].
\end{equation}
is a compact rectangle as well. We denote

\begin{align}
\Wst{x,C_{x_{\alpha_0}^{i_0},\ldots x_{\alpha_r}^{i_r}}}&=\Wst{x,D_i}\cap  C_{x_{\alpha_0}^{i_0},\ldots x_{\alpha_r}^{i_r}}\\
\Wut{x,C_{x_{\alpha_0}^{i_0},\ldots x_{\alpha_r}^{i_r}}}&=\Wut{x,D_i}\cap C_{x_{\alpha_0}^{i_0},\ldots x_{\alpha_r}^{i_r}}.
\end{align}

Observe that $\emph{im}\theta=\cup_{i,\alpha} C_{x_{\alpha}^i}$, and each one of these rectangles has diameter less than equal to $2C\del$, thus 
$$A_{x^i_{\alpha},x^{i'}_{\alpha'}}=1\Rightarrow f(C_{x_{\alpha}^i})\subset sat^c(D_{i'},r_0).$$

\esp

\begin{lemma}
For $\underline{a}\in [x^i_{\alpha}]$ it holds
\begin{enumerate}
\item[a)] $\theta\left( W^s_{loc}(\underline{a}) \right)=\Wst{\theta(\underline{a}), C_{x_{\alpha}^i}}$. 
\item[b)] $\theta\left( W^u_{loc}(\underline{a}) \right)=\Wut{\theta(\underline{a}), C_{x_{\alpha}^i}}$.
\end{enumerate}
\end{lemma}

\begin{proof}
By proposition \ref{brapre},  $\theta\left(  W^s_{loc}(\underline{a}) \right)\subset\Wst{\theta(\underline{a}),C_{x_{\alpha}^i}}$. Consider a point $x\in$ $\Wst{\theta(\underline{a}),C_{x_{\alpha}^i}}$\ and take $\underline{b}\in [x_{\alpha}^i]$\ such that $\theta(\underline{b})=x$. Now $\underline{c}=\langle \underline{a},\underline{b}\rangle \in W^s_{loc}(\underline{a})$\ and $$ \theta(\underline{c})=\langle\theta(\underline{a}),\theta (\underline{b})\rangle=\langle\theta( \underline{a}),x\rangle =x$$
since $x$\ in $\Wst{\theta \underline{a},D_i}$. Therefore $x\in\theta(W^s_{loc}(\underline{a}))$. The second part is similar.
\end{proof}

Next we analyse the behaviour of $\theta$ with respect to the shift dynamics: it will be established below that the natural Markov structure of the rectangles $[x^i_{\alpha}]$ with respect to $\sigma$ induces a similar property on $\emph{im}\theta$. First note:

\begin{lemma}\label{dynrectangulo}
Consider $\underline{a}\in [x^i_\alpha,x^{i'}_{\alpha'}]$. Then $proj_{D_{i'}}(f\theta\underline{a})=\theta\sigma\underline{a}$. 
\end{lemma}

\begin{proof}
Let $x=\theta(\underline{a})$. Observe that the sequence $(b_n:=a_{n+1}')_{n\in\mathbb{Z}}$\ $C\delta$-shadows $\sigma (\underline{a})$, and since it respects the center foliation we conclude $$\theta(\sigma\underline{a})=proj_{D_{i'}}(a_1'),$$
$$
fx\in \Wcloc{a_1'}
$$
hence the result. 
\end{proof}

\esp

\begin{proposition}\label{markovseq} Let $\underline{a}\in \Sigma_{A}$ and consider  $x=\theta(\underline{a}) , y=\theta\sigma\underline{a}, z=\theta\sigma^{-1}\underline{a}$. If $i,i',i''$ are such that $x\in D_i,y\in D_{i'},z\in D_{i''}$, then

\begin{align}
proj_{D_{i'}}(f\Wst{x,C_{a_0}})&\subset\Wst{y,C_{a_1}} \\
proj_{D_{i}}(f\Wut{z,C_{a_{-1}}})&\supset\Wut{x,C_{a_0}}
\end{align}
\end{proposition}

\begin{proof}
This is an immediate consequence of the two previous lemmas:
\begin{equation*}
\begin{split} proj_{D_{i'}}(f\Wst{x,C_{a_0}})& =proj_{D_{i'}}(f\Theta(W^s_{loc}(\underline{a})))=
\theta\sigma W^s_{loc}(\underline{a}) \\ 
&\subset\theta(W^s_{loc}(\sigma \underline{a}))=\Wst{y,C_{a_1}},
\end{split}
\end{equation*}
and likewise for the unstable part.
\end{proof}

\esp

\subsection{Bowen's refinement procedure.}  The reader may have observed before that we overdetermined the possible transitions determined by $A^1$. We did so in order to guarantee that every $\Delta^i_{\alpha}$ is contained in the image of $\phi$ (Cf. Proposition \ref{sobre}). There appears to be a trade-off between this ``surjectivity'' and the the over determinacy, related to the fact that we are using projections along centers, and those are difficult to control.

Now we are going to eliminate some unnecessary redundancy using a method of R. Bowen (compare \cite{Shub} pages 133-134) to cut the rectangles $C_{x^i_{\alpha}}$ into disjoint sub-rectangles, while maintaining the Markov property. For rectangles $C_{x^i_{\alpha}},C_{x^i_{\beta}}$ with non-empty intersection, define:

\begin{align*}
C_{x^i_{\alpha}:x^i_{\beta}}^1=\{x\in C_{x^i_{\alpha}}:\Wst{x,C_{x^i_{\alpha}}}\cap C_{x^i_{\beta}}\neq \emptyset,
\Wut{x,C_{x^i_{\alpha}}}\cap C_{x^i_{\beta}}\neq \emptyset\} \\
C_{x^i_{\alpha}:x^i_{\beta}}^2=\{x\in C_{x^i_{\alpha}}:\Wst{x,C_{x^i_{\alpha}}}\cap C_{x^i_{\beta}}\neq \emptyset,
\Wut{x,C_{x^i_{\alpha}}}\cap C_{x^i_{\beta}}= \emptyset\} \\
C_{x^i_{\alpha}:x^i_{\beta}}^3=\{x\in C_{x^i_{\alpha}}:\Wst{x,C_{x^i_{\alpha}}}\cap C_{x^i_{\beta}}= \emptyset, 
\Wut{x,C_{x^i_{\alpha}}}\cap C_{x^i_{\beta}}\neq \emptyset\} \\
C_{x^i_{\alpha}:x^i_{\beta}}^4=\{x\in C_{x^i_{\alpha}}:\Wst{x,C_{x^i_{\alpha}}}\cap C_{x^i_{\beta}}= \emptyset,
\Wut{x,C_{x^i_{\alpha}}}\cap C_{x^i_{\beta}}= \emptyset\}. 
\end{align*}

It is not too hard to show that each $C_{x^i_{\alpha}:x^i_{\beta}}^m, m=1,\ldots, 4$ is a rectangle, and furthermore
\begin{enumerate}
\item[i)] $C_{x^i_{\alpha}:x^i_{\beta}}^1=C_{x^i_{\alpha}}\cap C_{x^i_{\beta}}$,
\item[ii)] the sets $C_{x^i_{\alpha}:x^i_{\beta}}^1,C_{x^i_{\alpha}:x^i_{\beta}}^1\cup C_{x^i_{\alpha}:x^i_{\beta}}^2,C_{x^i_{\alpha}:x^i_{\beta}}^1\cup C_{x^i_{\alpha}:x^i_{\beta}}^3$ are closed in $\cup_{i} D_{i}$, and
\item[iii)] $\cup_{m=1}^4  C_{x^i_{\alpha}:x^i_{\beta}}^m= C_{x^i_{\alpha}}$.
\end{enumerate}

Now define
\begin{align*}
R_{x^i_{\alpha}:x^i_{\beta}}^{m}&=cl\ C_{x^i_{\alpha}:x^i_{\beta}}^m,\ m=1,\ldots, 4\\
F_{x^i_{\alpha}:x^i_{\beta}}&=\cup_{m=1}^4 \partial C_{x^i_{\alpha}:x^i_{\beta}}^m,
\end{align*}
where the closure and the boundary are taken with respect to the induced topology of the disc $D_{i}$. We have that $\{R_{x^i_{\alpha}:x^i_{\beta}}^{m}\}_{m=1}^4$ is a covering of $C_{x^i_{\alpha}}$ by closed rectangles whose interiors (with respect to the induced topology)\ are disjoint. See Lemma 10.24 in \cite{Shub}.

\esp
Consider the set 
\begin{equation*}
\mathcal{N}:=\emph{im}\theta\setminus \bigcup_{i,\alpha,\beta}F_{x^i_{\alpha}:x^i_{\beta}},
\end{equation*}
which is relatively open and also dense in $im\theta$, due to Baire's catetegory theorem. For $x\in \mathcal{N}$ define:

\begin{align*}
\mathcal{J}(x)&=\{C_{x^i_{\alpha}}:x\in C_{x^i_{\alpha}}\}\\
\mathcal{J}^{*}(x)&=\{C_{x^i_{\beta}}:\exists C_{x^i_{\alpha}}\in\mathcal{J}(x), C_{x^i_{\alpha}}\cap C_{x^i_{\beta}}\neq\emptyset\}\\
P(x)&=\bigcap\{int\, R_{x^i_{\alpha}:x^i_{\beta}}^m:  C_{x^i_{\alpha}}\in\mathcal{J}(x), C_{x^i_{\beta}}\in\mathcal{J}^{*}(x),
x \in C_{x^i_{\alpha}:x^i_{\beta}}^m\}.
\end{align*}

\esp

\begin{lemma}\label{rectanglePx}
Each $P(x)$\ is an open rectangle and  $x,x'\in \mathcal{N},\ P(x)\cap P(x')\neq\emptyset \Rightarrow P(x)=P(x')$. In particular there are finitely many rectangles $P(x)$.
\end{lemma}

\begin{proof}
It follows easily that each $P(x)$ is an open rectangle. Since $\mathcal{N}$\ is dense in $\emph{im}\theta$, it
it suffices to show that  $$x'\in P(x)\cap \mathcal{N}\Rightarrow P(x)=P(x')$$
and for this it will be enough to show $\mathcal{J}(x)=\mathcal{J}(x')$, as the rectangles $C_{x^i_{\alpha}:x^i_{\beta}}^m$\ have disjoint interiors.

It is immediate that $\mathcal{J}(x)\subset\mathcal{J}(x')$; for the other inclusion take $C_{x^i_{\beta}}\ni x'$ and consider $C_{x^i_{\alpha}}\in\mathcal{J}(x)$\  such that
$C_{x^i_{\alpha}}\cap C_{x^i_{\beta}}\neq\emptyset$. By checking the possibilities, one sees that necessarily $x\in C_{x^i_{\alpha}:x^i_{\beta}}^1=C_{x^i_{\alpha}}\cap C_{x^i_{\beta}}$, and thus  $x\in  C_{x^i_{\beta}}$\ as well, i.e. $\mathcal{J}(y)\subset\mathcal{J}(x)$.
\end{proof}

\esp
We now discuss the dynamics of the rectangles $P(x)$.

\begin{definition}
We say that $D_i'$ is an allowed future for $P(x)\subset D_i$ if there exists $\alpha,\alpha'$ such that $C_{x^i_{\alpha}}\in \mathcal{J}(x)$ and $A_{x^i_{\alpha},x^{i'}_{\alpha'}}=1$.
\end{definition}

\begin{lemma}\label{rectangulosB}
Suppose that $x\in C_{x^i_{\alpha},x^{i'}_{\alpha'}}$, and $y=proj_{D_i'}(fx)\in C_{x^{i'}_{\beta}}$. Then there exists $\underline{b}\in\Sigma_A$ satisfying
\begin{enumerate}
\item $\underline{b}\in [x^i_{\alpha},x^{i'}_{\beta}]$,
\item $\theta(\underline{b})=x, \theta(\sigma\underline{b})=y$.
\end{enumerate}
\end{lemma}

\begin{proof}
Observe that by definition of $A$ one has $A_{x^i_{\alpha},x^{i'}_{\beta}}=1$. By hypothesis there exist $\underline{a},\underline{c}\in\Sigma_A$ such that
\begin{enumerate}
\item[i)] $a_0=x^i_{\alpha},c_0=x^{i'}_{\beta}$,
\item[ii)]$\theta(\underline{a})=x,\theta(\underline{c})=y$.
\end{enumerate}
Consider the sequence defined by $$ b_n=\begin{cases}
a_n&\text{if }n\leq 0 \\
c_{n-1}&\text{if }n\geq 1,
\end{cases}$$
and observe that $\underline{b}\in\Sigma_A$. Let $z=\theta\underline{b}$: we claim that $z=x$.  To see this observe that since $\underline{b}\in W^u_{loc}(\underline{a}), z\in\Wut{x,C_{a_0}}$. Furthermore $\sigma \underline{b} \in  W^s_{loc}(\underline{c})$, so we get $proj_{D_{i'}}(fz)\in\Wst{y,C_{c_0}}$ and thus $fz\in \Wcs{fx}$. We conclude  $z\in\Wcsloc{x}\cap\Wut{x,C_{a_0}}=\{x\}$. 
\end{proof}

\begin{proposition}\label{dynamicaP}
Let $x,x'\in \mathcal{N}\cap D_i$ be such that
\begin{enumerate}
\item $P(x)=P(x')$.
\item $x'\in \Wst{x,D_{i}}$.
\item $i'$ is an allowed future for $P(x)$.
\end{enumerate}
Consider the points $y=proj_{D_{i'}}(fx),y'=proj_{D_{i'}}(fx')$, and suppose that $y,y'\in\mathcal{N}$. Then $P(y)=P(y')$
\end{proposition}

\begin{proof}
We will first establish that $\mathcal{J}(y)=\mathcal{J}(y')$, and for this we note that due to symmetry it suffices to show only one inclusion. Suppose that $y\in C_{x^{i'}_{\beta}}$: since $i'$ is an allowed future for $P(x)$ the previous Lemma allow us to find a sequence $\underline{a}\in [x^{i}_{\alpha},x^{i'}_{\beta}]$ satisfying $\theta(\underline{a})=x,\theta(\sigma\underline{a})=y$. By Proposition \ref{markovseq},
$$
y'\in proj_{D_{i'}}(f\Wst{x,C_{a_0}})\subset\Wst{y,C_{a_1}}\subset C_{x^{i'}_{\beta}},
$$
i.e. $C_{x^{i'}_{\beta}}\in \mathcal{J}(y')$, hence $\mathcal{J}(y)\subset \mathcal{J}(y')$ as we wanted to show.

It then follows that $\mathcal{J}^{\ast}(y)=\mathcal{J}^{\ast}(y')$, and 
$$
\Wst{y,C_{x^{i'}_{\beta}}}\cap C_{x^{i'}_{\beta'}}\neq \emptyset\Leftrightarrow \Wst{y',C_{x^{i'}_{\beta}}}\cap C_{x^{i'}_{\beta'}}\neq \emptyset.
$$

We want to show that the same is true for the unstable manifolds, and thus conclude $P(y)=P(y')$. Consider a point $w\in\Wut{y,C_{x^{i'}_{\beta}}}\cap C_{x^{i'}_{\beta'}}$\ and take a sequence $\underline{a}\in [x^{i}_{\alpha},x^{i'}_{\beta}]$ as above. We have 
$$
proj_{D_{i'}}(f\Wut{x;C_{a_0}})\supset \Wut{y;C_{a_1}},
$$
hence there exists $z\in \Wut{x;C_{a_0}}$ such that $proj_{D_{i'}}(fz)=w$. The set $C_{a_0}$ is a rectangle, and since $P(x)=P(x')$ we get $u=\langle z,x' \rangle\in \Wut{x';C_{a_0}}$. As $\theta$ is bracket preserving and $C_{a_1}$ is a rectangle,
$$
u'=proj_{D_{i'}}(fu)=\langle w,y' \rangle \in C_{a_1}.
$$
On the other hand,  $z\in C_{x^{i}_{\alpha}},w\in C_{x^{i'}_{\beta'}}$ and thus there exists
$\underline{c}\in\Sigma_A$\ such that
$$
\theta(\underline{c})=z,\theta(\sigma\underline{c})=w,c_0=x_{\alpha}^{i},c_1=x^{i'}_{\beta'}.
$$
Using that $proj_{D_{i'}}(f\Wst{z;C_{c_0}})\subset \Wst{w,C_{c_1}}$ we conclude $u'\in C_{c_1}$, i.e.
$$
u'\in \Wut{y,C_{x^{i'}_{\beta}}}\cap C_{x^{i'}_{\beta'}}.
$$
We have shown that
$$
\Wut{y,C_{x^{i'}_{\beta}}}\cap C_{x^{i'}_{\beta'}}\neq \emptyset\Rightarrow \Wut{y,C_{x^{i'}_{\beta}}}\cap C_{x^{i'}_{\beta'}}\neq \emptyset
$$
and by symmetry we have the other implication. This completes the proof of the Proposition.
\end{proof}

\esp

\subsection{The Markov Partition for the transversal.}

Lemma \ref{rectanglePx} implies that the family $\mathcal{R}=\{R=cl P(x):x\in\mathcal{N}\}$ is finite, say $\mathcal{R}=\{R_1,\ldots R_s\}$. Each $R_m$ is a compact rectangle, with $R_m=cl\, int\, R_m$. Furthermore, if $P(x)\neq P(y)$ then $P(x)\cap P(y)=\emptyset$, and since these are open sets a simple topological argument implies that rectangles in $\mathcal{R}$ have pairwise disjoint interiors. We extend the notion of `allowed future' to rectangles in $\mathcal{R}$.

\begin{definition}
If $R_{\mu}\in \mathcal{R},D_i'$ allowed future for $R_{\mu}$ we define the map $\phi_{\mu,i'}^+:R_{\mu}\mapsto D_{i'}$ by
\begin{equation}
\phi_{\mu,i'}^+(x)=proj_{D_{i'}}(fx).
\end{equation}
\end{definition}

\begin{proposition}[Markov Property of $\mathcal{R}$]\label{dynamicaR}
Let $R_{\mu},R_{\nu}\in \mathcal{R}$ with $D_{i'}$ allowed future for $R_{\mu}$, and suppose that $x\in int\, R_{\mu}, \phi_{\mu,i'}^+(x)\in int\, R_{\nu}$. Then
\begin{align}
\phi_{\mu,i'}^+(\Wst{x,R_{\mu}})\subset \Wst{\phi_{\mu,i'}^+(x),R_{\nu}}\\
\phi_{\mu,i'}^+(\Wut{x,R_{\mu}})\supset \Wut{\phi_{\mu,i'}^+(x),R_{\nu}}
\end{align}
\end{proposition}

\begin{proof}
For a rectangle $R_{\eta}$ denote
$$
\partial^s R_{\eta}:=\{y:\in R_{\eta}:\Wst{x;R_{\eta}}\cap int\, R_{\eta}\}.
$$ 

Now take $x\in R_{\mu}\cap \mathcal{N}$ such that $y=\phi_{\mu,i'}(x)\in R_{\nu}\cap \mathcal{N}$. Then $R_{\mu}=cl P(x),R_{\nu}=cl P(y)$, and
$$
\Wst{x,R_{\mu}}=cl \Wst{x;P(x)},\ \Wst{y,R_{\nu}}=cl \Wut{y;P(x)}.
$$
By Proposition \ref{dynamicaP}
$$
\phi_{\mu,i'}^+(\Wst{x,P(x)})\subset \Wst{y;P(y)},
$$
and thus by continuity
$$
\phi_{\mu,i'}^+(\Wst{x,R_{\mu}})\subset \Wst{\phi_{\mu,i'}^+(x),R_{\nu}}.
$$

Points $x$ as above are dense in $\cup_{\eta} R_{\eta}$: standard arguments permit then to extend the conclusion for any point $x'\in int\, R_{\mu}, \phi_{\mu,i'}^+(x')\in int\, R_{\nu}$. Details can be found for example in \cite{Shub} page 137. Likewise for the unstable part.
\end{proof}

\esp

\begin{definition}
The disc $D_k$ is an allowed past of $R_{\mu}\subset D_i$ if $D_i$ is an allowed future of some rectangle
$R_{\nu}\subset D_k$ with $\phi_{\nu,i}(int R_{\nu})\cap int R_{\mu}\neq \emptyset$. In this case
define the map $\phi_{\mu,k}^{-}:R_{\mu}\mapsto D_{k}$ by
\begin{equation}
\phi_{\mu,k}^{-}(x)=proj_{D_{k}}(f^{-1}x).
\end{equation}
\end{definition}

\esp

The following is immediate from Proposition \ref{dynamicaR}.

\begin{corollary}\label{CorodynamicaR}
Let $R_{\mu},R_{\nu}\in \mathcal{R}$ with $D_{k}$ allowed past for $R_{\mu}$, and suppose that $x\in int\, R_{\nu}$, $\phi_{\mu,k}^{-}(x)\in int\, R_{\mu}$. Then
\begin{equation}
\phi_{\mu,k}^{-}(\Wut{x,R_{\mu}})\subset \Wut{\phi_{\mu,k}^{-}(x);R_{\nu}}
\end{equation}
\end{corollary}

We subsume the results of this section in the following Theorem.

\begin{theorem}\label{MarkovTransversal}
Given $\delta>0,\rho>0$ there exist a family $\mathcal{D}=\{D_1,\ldots, D_N\}$of $s+u$ dimensional (embedded) discs transverse to \Fc, a family $\mathcal{R}=\{R_1,\ldots,R_s\}$ and a function $af:\mathcal{R}\rightarrow 2^{\mathcal{D}}$  satisfying the following.

\begin{enumerate}
\item The family $\mathcal{D}$ is pairwise disjoint.
\item Each $R_{\mu}$ is a rectangle contained in some disc $D_i$, and with respect to the relative topology of $D_i$ it holds $R_{\mu}=cl\, int R_{\mu}$. Moreover $diam R_{\mu}<\rho$.
\item If $R_{\mu},R_{\nu}$ are contained in the same disc $D_i$ and their interiors share a common point, then $R_{\mu}= R_{\nu}$.
\item $M=\cup_{\mu=1}^s sat^c(R_{\mu},\delta)$.
\item Suppose that $R_{\mu}\subset D_i,R_{\nu}\in D_{i'}$ satisfy  $D_{i'}\in af(R_{\mu})$. If  $x\in int\, R_{\mu}, \phi_{\mu,i'}^+(x)\in int\, R_{\nu}$, then it holds
\begin{enumerate}
\item $\phi_{\mu,i'}^+(\Wst{x,R_{\mu}})\subset \Wst{\phi_{\mu,i'}^+(x),R_{\nu}}$.
\item $\phi_{\mu,i'}^{-}(\Wut{\phi_{\mu,i'}^+(x),R_{\nu}})\subset \Wut{x,R_{\mu}}$
\end{enumerate}
where $\phi_{\mu,i'}^+=proj_{D_{i'}}\circ f,\phi_{\mu,i'}^-=proj_{D_i}\circ f^{-1}$.
\end{enumerate}
\end{theorem}

In synthesis, the rectangles $R_1,\ldots R_s$ are proper, with disjoint interiors and satisfy the Markov property for suitable choices of the future. Observe nonetheless that these rectangles do not have a canonically defined future, nor a (canonically defined) past. Theorem A is  proved.

\section{Symbolic representation of the transversal.}

Consider the subshift of finite type $\Sigma_S$ determined by the matrix $S$ where
$$
S_{\mu,\nu}=1\Leftrightarrow \nu\subset D_{i'},i'\in af(R_{\mu})\text{ and } \phi_{\mu,i'}^+(int R_{\mu})\cap int R_{\nu}\neq\emptyset.
$$

For $\underline{a}\in \Sigma_S$ and $l\leq l'\in \mathbb{Z}$ denote $a[l,l']=a_l,\ldots,a_{l'}$. Observe then for the string $a[l,l']$ the rectangles $R_{a_l},\ldots R_{a_{l'}}$ are ``linked'', in the sense
that if $D_{i_{\eta}}$ denotes the disc containing $R_{\eta}, \eta=l,\ldots,l'$, there exists a non-empty relatively open set $V_{a[l,l']}\subset R_{a_l}$ such that the map
$$\
\phi_{a[l,l']}^+=\phi_{a_{l'},i_{l'}}^+\circ,\ldots,\phi_{a_{l+1},i_{l+1}}^+:R_{a_l}\rightarrow R_{a_{l'}} 
$$
is  well defined and continuous. Similarly,
$$
\phi_{a[l,l']}^-:W_{a[l,l']}\subset R_{a_{l'}}\rightarrow R_{a_l}
$$
is constructed using the maps $\phi^{-}_{a_{\eta},i_{\eta}}$.

These considerations allow us to define
\begin{equation}
h^{s}(\underline{a})=\bigcap_{n\geq 0}(\phi_{a[o,n]}^{+})^{-1}(R_{a_n})=\bigcap_{n\geq 0} V_{a[0,n]}
\end{equation}
\begin{equation}
h^{u}(\underline{a})=\bigcap_{n\geq 0}(\phi_{a[-n,0]}^{-})^{-1}(R_{a_{-n}})=\bigcap_{n\geq 0} W_{a[-n,0]} 
\end{equation}
\begin{equation}
h(\underline{a})=h^{+}(\underline{a})\cap h^{-}(\underline{a}).
\end{equation}

\esp

\begin{lemma}\label{haches}
If $\underline{a}\in\Sigma_S$, there exists $x\in R_{a_0}$ such that
\begin{enumerate}
 \item $h^s(\underline{a})=\Wst{x,R_{a_0}}$
 \item $h^{u}(\underline{a})=\Wut{x,R_{a_0}}$
 \item $h(\underline{a})=x$.
 \end{enumerate}
\end{lemma}

\begin{proof}
Part 5 of Theorem \ref{MarkovTransversal} imply that there exists a subset $T\subset R_{a_0}$ such that  $\cap_{n\geq 0} V_{a[0,n]}=\cup_{x\in T} \Wst{x,R_{a_0}}$. As $\Fc$ is plaque expansive, $\Fcs$ is plaque expansive as well (Cf. Lemma \ref{plaqueexpother}), and thus for some $x\in R_{a_0}, T=\Wst{x,R_{a_0}}$. The second part is similar and the third is consequence of the first two.
\end{proof}

We discuss now continuity of these maps. Continuity of $h:\Sigma_S\rightarrow \cup_{\eta} R_{\eta}$ will be understood with respect to the product metric of $\Sigma_S$. In the case
of $h^{+},h^{-}$ we arbitrarily pick points $x_{\mu}\in int R_{\mu}$: continuity of $h^{s}$ will be understood as continuity
of the map $$\Sigma_S\xrightarrow[]{h^{s}}\cup_{\mu} R_{\mu}\rightarrow \cup_{\mu}\Wut{x_{\mu},R_{\mu}},$$ where the last map collapses each $R_{\mu}$ to $\Wu{x_{\mu},R_{\mu}}$ along the partition $\{\Ws{x,R_{\mu}}\}_{x\in R_{\mu}}$. Similarly for $h^{u}$.

\esp

\begin{proposition}\label{projcont}
 The maps $h^{s},h^{u},h$ are continuous, where $\Sigma_S$ is equipped with the product metric. If furthermore $\Fc$ is Lipschitz, then they are Lipschitz continuous if $\delta$ is sufficiently small. 
\end{proposition}

\begin{proof}
It suffices to establish continuity of $h^{s}$. Since $\mathcal{R}$ is finite and $\delta$ is fixed we have that for $n\geq0$ it holds
$$
r_n:=\inf\{diam(\Wut{x,V_{a[0,n]}}):\underline{a}\in\Sigma_S, x\in V_{a[0,n]}\}>0.
$$

Now assume by contradiction that $h^{s}$ is not continuous. Then there exists a converging sequence $\underline{a}^N\xrightarrow[N\mapsto\oo]{}\underline{a}$
in $\Sigma_S$ such that $x_N=h^{s}(\underline{a}^N)\xrightarrow[n\mapsto\oo]{} x\neq h^{s}(\underline{a})$. Fix $n$ and take $N_0$ such that for $N\geq N_0,d(x_N,x)<r_n$. Then by definition of $h^{s}$
$$
x\in V_{a^N[0,n]},
$$
and since $\underline{a}^N\xrightarrow[N\mapsto\oo]{}\underline{a}$, we deduce that $x\in V_{a[0,n]}$ which in turn implies that $x\in\cap_{n\geq 0} V_{a[0,n]}$. This contradicts the fact that $\Fcs$ is plaque expansive. 

If we further assume that $\Fc$ is Lipschitz, we can choose $\delta$ sufficiently small such that for every holonomy transport $h$ associated to an $(\epsilon,\delta)$  family, the map $h\circ f$ contracts stable distances by a factor of $0<k<1$. Then  for every $\underline{a}\in \Sigma_S,n\geq0$,
$$
sup\{diam(\Wu{x,V_{a[0,n]}}):x\in V_{a[0,n]}\}\leq k^n sup\{diam(\Wu{x,R_{a_0}}):x\in R_{a_0}\},  
$$
and it follows that $h^{s}$ is Lipschitz. 
\end{proof}

\begin{remark}
For a general partially hyperbolic diffeomorphism, the center foliation is only H\"older continuous \cite{HolFol}. Nonetheless, establishing that 
$h,h^{s},h^{u}$ are H\"older in general seems difficult without assuming some condition that forzes the diameter of the sets $V$ to
diminish after every iteration. The argument seems similar to the one necessary to establish plaque expansiveness in general, which is at the moment of writing, still unknown. 
\end{remark}

We now investigate the dynamics of the subshifts. Recall that given a topological space $X$, a pseudo-group on $X$ is a set $\mathcal{G}=\{g:d(g)\rightarrow r(g)\}$ of local homeomorphisms between proper sets\footnote{In the literature, the notion of pseudo-group
is usually defined using open sets instead of proper ones, but the extension to the later case is straightforward} of $X$ such that
\begin{enumerate}
 \item $Id:X\rightarrow X\in \mathcal{G}$
 \item $g\in \mathcal{G}\Rightarrow g^{-1}\in\mathcal{G}$.
 \item $g:d(g)\rightarrow r(g),g':d(g')\rightarrow r(g')\in \mathcal{G}$ with $d(g')\cap r(g))$ proper then 
 $$
 g'\circ g:g^{-1}d(g')\rightarrow g'(r(g))\in \mathcal{G}.
 $$
 \item $g\in \mathcal{G}, U\subset d(g)$ proper implies $g|:U\rightarrow g(U)\in\mathcal{G}$.
\item Suppose that $U,V\subset X$ are proper and $g:U\rightarrow V$ is a homeomorphism. Suppose that there exist a family $\{g_i:d(g_i)\rightarrow r(g_i)\}_i\subset \mathcal{G}$
such that $\cup_i d(g_i)-U, g|d(g_i)=g$. Then $g\in\mathcal{G}$.
\end{enumerate}
See for example \cite{FolMeasPres}. Note that there is a natural action $\mathcal{G}\curvearrowright X$.

\esp

Let $\Gamma(\mathcal{R})=\cup_{\mu}R_{\mu}$ and consider the pseudo-group $\mathcal{G}(\mathcal{R})$ of homeomorphisms of $\Gamma(\mathcal{R})$ generated by $\{\phi_{\mu,i'},\phi_{\mu,i'}^-:i'\in af(R_{\mu})\}$. Our previous discussions can be summarized by the following (Cf. Corollary \ref{coroA}).

\esp

\begin{theorem}
There exists a continuous surjective map $h:\Sigma_S\rightarrow \Gamma(\mathcal{R})$ which semi-conjugates the action of the shift $\sigma\curvearrowright\Sigma_S$
with the natural action $\mathcal{G}(\mathcal{R})\curvearrowright \Gamma(\mathcal{R})$.
\end{theorem}

Similar, conclusions can be drawn for the dynamics of the center-stable and center-unstable plaques using the maps $h^s,h^u$.

\esp

The symbolic model presented is not well behaved under center holonomy. This is to be expected since there is no reason for this holonomy to have any hyperbolic behavior. Take for example $f:\mathbb{T}^3\rightarrow \mathbb{T}^3$ a product of an Anosov diffeomorphism of $\mathbb{T}^2)$ and the identity of $S^1$: here the center foliation is a circle fibration and thus the center holonomies are just the identity. Even when there is some hyperbolicity associated to the center foliation it is not clear how to deal with the interchange of leaves. The simplest case would be when $f:M\rightarrow M$ is a regular element in an (abelian) Anosov action. 

\esp

\textbf{Question:} Let $\mathcal{R}^k\curvearrowright M$ Anosov action with regular element $f$. Can one get an analogue of Theorem A where the rectangles are invariant by all elements of $G$, or at least by a subgroup strictly larger than $<f>$? 

\esp

The plausible answer is No, at least without imposing some strong conditions on the action. Otherwise it seems that one would get many different invariant measures for the action, something very unusual. 

\section*{Acknowledgements}

The present results are a generalization of the ones obtained by the author his MSc project under the supervision of Mike Shub. I am very grateful to him for all the ideas and encouragement that he gave me. I have also benefited a lot from conversations wit Charles Pugh, Enrique Pujals and Federico Rodriguez-Hertz. To all of them my most sincere thanks

\bibliographystyle{alpha}
\bibliography{biblio}
\end{document}